\documentclass{amsart}

\usepackage{amssymb,amsmath,cite}

\usepackage{amsfonts}
\usepackage{enumerate}
\usepackage{mathrsfs}
\usepackage{amscd}
\usepackage{amsmath}
\usepackage{latexsym}
\usepackage{amssymb}
\usepackage{amsthm}
\usepackage{graphicx}
\usepackage{color}
\usepackage{cite}
\usepackage{bm}

\newcommand{\ve}{\varepsilon}

\newtheorem{definition}{Definition}[section]
\newtheorem{lemma}{Lemma}[section]
\newtheorem{theorem}{Theorem}[section]

\theoremstyle{remark}
\newtheorem{remark}[theorem]{Remark}

\begin{document}

\title[well-posedness  and   blowup of  the geophysical boundary layer problem]{ Well-posedness  and   blowup of  the geophysical boundary layer problem}

\author{X. Wang}
\address{Xiang Wang
\newline\indent
School of Mathematical Sciences, Shanghai Jiao Tong University,
Shanghai, P. R. China}
\email{wang\_xiang@sjtu.edu.cn}

\author{Y.-G. Wang}
\address{Ya-Guang Wang
\newline\indent
School of Mathematical Sciences, MOE-LSC and SHL-MAC, Shanghai Jiao Tong University,
Shanghai, 200240, P. R. China}
\email{ygwang@sjtu.edu.cn}

\maketitle

\date{}

\begin{abstract}
Under the assumption that the initial velocity and outflow velocity are analytic in the horizontal variable,  the local well-posedness of the geophysical boundary layer problem is obtained by using energy method in the weighted Chemin-Lerner spaces. Moreover, when the initial velocity and outflow velocity  satisfy certain condition on a transversal plane, it is proved that the $W^{1,\infty}-$norm of any smooth solution decaying exponentially in the normal variable to the geophysical boundary layer problem blows up in a finite time.
\end{abstract}
~

%
%
%
%
%

\section{Introduction }
In this paper, we  consider the following  initial boundary value problem in the domain $Q_T=\{0< t< T, x\in \mathbb{R}, y> 0\}$,
\begin{equation}\label{OGP} \left\{\!\!
\begin{array}{lc}
\partial_t u+u\partial_x u+v\partial_y u+\int^y_{+\infty}(u-U)dy'-\partial_{y}^2u=\partial_t U+U\partial_x U, &\\[8pt]
\partial_x u+\partial_{y}v=0, &\\[8pt]
u|_{t=0}=u_0(x, y), &\\[8pt]
(u, v)|_{y=0}=(0, 0),  ~~\lim \limits_{y \to
+\infty} u(t,x,y)=U(t, x),
\end{array}
\right.
\end{equation}
where $(u,v)$ is the velocity field, and $U(t,x)$ is the  tangential velocity   of  the outer flow.  

 The problem \eqref{OGP} describes the oceanic current near the western coast, it can be derived from the beta plane approximation model of  the oceanic current motion  at midlatitudes under the action of wind and  the Coriolis force in the large Reynolds number and beta parameter limit. By properly scaling in certain geophysical regime, and omitting the bottom friction and topography, the beta plane approximation of the oceanic current can be described by the following two dimensional homogeneous model (see \cite{De-G,Ped}) in $\{x\in \mathbb{R}, Y>0\}$,
 \begin{equation}\label{nsc00}
\left\{
\begin{array}{l}
 \partial_t {\bf U}+ {\bf U}\cdot\nabla {\bf U}-\beta x{\bf U}^{\bot}+\nabla \Pi-Re^{-1}\Delta {\bf U}=\beta\tau ,\\[2mm]
 \text{div} ~{\bf U}=0,\\[2mm]
 {\bf U}|_{Y=0}=0
\end{array}
\right.
\end{equation}
 where the Cartesian-like coordinates $(x,Y)$  represent latitude and longitude respectively, and ${\bf U}=(U, V)^T$, $\Pi$,  $Re$ and $\beta$  are the velocity, the pressure of  fluid, the Reynolds number and the beta-plane parameter respectively,  $\tau=(\tau_1,\tau_2)^T$ is the shear tensor created by wind, and $-x{\bf U}^{\bot}$ represents the effect of the Coriolis force created by rotation with ${\bf U}^{\bot}=(-V,U)^T$.  When $\text{Re}=\beta^2$, for which the inertial force, the Coriolis force and viscous friction have the same order in boundary layer, by multi-scale analysis it is known that as $\epsilon={\rm Re}^{-1}\to 0$, the solution of
 \eqref{nsc00} near $\{Y=0\}$ behaves as
 $$
 \begin{cases}
 U(t,x,Y)=u(t,x,\frac{Y}{\sqrt{\epsilon}})+o(1)\\
 V(t,x,Y)=\sqrt{\epsilon}v(t,x,\frac{Y}{\sqrt{\epsilon}})+o(\sqrt{\epsilon})
 \end{cases}
 $$
 with $(u(t,x,y), v(t,x,y))$ satisfying the geophysical boundary layer problem \eqref{OGP}. More detail of the derivation can be found in  \cite{W-WW,G-W}.

For the two-dimensional incompressible Navier-Stokes equations with non-slip boundary condition, Prandtl introduced in \cite{P} that in the small viscosity limit, the flow near the boundary is described
by a problem similar to \eqref{OGP} without the
integral term, which is called the Prandtl equation. Under the monotonicity assumption, $u_y>0$, the well-posedness of the Prandtl equation is established in \cite{R-W-X-Y,L-W-Y,O-S,M-W,X-Z} and references therein.
On the other hand, without the monotonicity condition, the well-posedness of the Prandtl equation was obtained in the analytic class and the Gevery class, cf. \cite{Caf-S,L-C-M,Z-Z,L-Y,Di-G}, and the blowup of the Sobolev norm of  solutions in a finite time was given in \cite{E-E,Kuk-V-W} for certain class initial data.

Certain formal discussion on the boundary layers of geophysical fluids can be found in \cite{Ped}. For the system \eqref{nsc00} with an additional bottom friction term, the behavior of the Munk layers and Stommel layers was given in \cite{De-G}. Recently, the well-posedness of a two-dimensional steady geophysical boundary layer problem was studied in \cite{DA-P}.

The aim of this paper is to study the local well-posedness  and  blowup of solutions to the unsteady geophysical boundary layer problem \eqref{OGP}.
  Compared with the classical Prandtl equation,  there has an additional integral term in \eqref{OGP}. To deal with this integral term, we shall study this problem in a weighted function space with respect to  the normal variable.
  By developing the energy method given  in \cite{Z-Z,Zhu-W} and estimating the additional integral term, we shall obtain the well-posedness of the geophysical boundary layer problem \eqref{OGP} when both of the initial data and the outer flow velocity are analytic in the tangential variable.
   On the other hand,  in a way similar to that given in \cite{Kuk-V-W},  by constructing a Lyapunov functional we shall deduce that the $W^{1,\infty}-$norm of the solution to the geophysical boundary layer problem \eqref{OGP} must blow up in a finite time for certain class of initial data and outer flow, which shows that in general, the above analytic solution exists only locally in time.
  Due to the integral term, we shall require  that the initial data and outer flow satisfy
\begin{equation}\label{Blowupcondition13}
u_{0x}(0,y)\leq U_{x}(0,0)
\end{equation}
such that the integral term  in \eqref{OGP} could keep the sign unchanged. It is interesting to see that this integral term  has a sensitive effect on the formation of singularity when we study  the blowup mechanism of the solution to this problem. Especially, in the case of the outer flow tangential velocity being zero identically, we get that the blowup of the solution always occurs in a finite time for any nonzero initial data, but, for which the classical Prandtl equation has a almost global solution (\cite{I-V, Z-Z}).    

  To state our main results, as in \cite{Z-Z},
 we introduce
 $$\phi(t,y)={\rm Erf}\left(\frac{y}{\sqrt{4(t+1)}}\right) ~~\text{with}~~{\rm Erf}(y)=\frac{2}{\sqrt{\pi}}\int_{0}^{y}e^{-z^2}dz,$$
to homogenize the condition of $u$ at infinity given in \eqref{OGP}.
Obviously, $\phi(t,y)$ is a solution to the problem
\begin{equation*}\label{phi}
\left\{
\begin{array}{l}
\partial_t \phi -\partial_{y}^{2}\phi =0,\\
\phi\big|_{y=0}=0,~~\lim\limits_{y\to +\infty}\phi(t,y)=1,\\
 \phi\big|_{t=0}={\rm Erf}(\frac{y}{2}).
\end{array}
\right.
\end{equation*}
Let $u^s=U\phi$ and $w = u-u^s$. From \eqref{OGP}, we know that  $w$ satisfies the following problem
\begin{equation}\label{homGP}
\left\{
\begin{array}{l}
\partial_t w +(w+u^s)\partial_{x}w+w\partial_{x}u^s-\int_{0}^{y}\partial_{x}(w+u^s)dy'
\partial_{y}(w+u^s)+\int_{+\infty}^{y}w dy'\\ [8pt]
~~~~~~~-\partial_{y}^{2}w
=(1-\phi)(\partial_t U+(1+\phi)U\partial_x U)-\int_{y}^{+\infty}U(1-\phi)dy',\\[8pt]
w\big|_{y=0}=0,~~\lim\limits_{y\to +\infty}w=0, \\[8pt]
w\big|_{t=0}=w_0(x,y)\triangleq u_0(x,y)-U(0,x)Erf(\frac{y}{2}).
\end{array}
\right.
\end{equation}

By applying the Littlewood-Paley theory, we shall obtain the existence and uniqueness of a solution to the problem \eqref{homGP} in the weighted Chemin-Lerner spaces, when the initial data and outflow velocity are analytic in $x\in \mathbb{R}$.
 \begin{theorem}\label{existence}
 For a given  $T_{0}>0$,   assume that the initial velocity $w_{0}(x,y)$ and the outflow velocity $U(t,x)$ are
 analytic in $x\in \mathbb{R}$, and
 $$e^{\langle D\rangle}w_0\in B_{\psi_{0}}^{\frac{1}{2},0}$$ and
 $$ e^{\langle D\rangle}U\in \tilde{L}^{\infty}_{T_{0}}(B^{\frac{1}{2}})
 \cap
 \tilde{L}^{\infty}_{T_{0}}(B^{1})
  \cap
  \tilde{L}^{\infty}_{T_{0}}(B^{\frac{3}{2}}), \quad
  e^{\langle D\rangle}U_t \in \tilde{L}^{\infty}_{T_{0}}(B^{\frac{1}{2}}) , $$
 where the spaces given at above will be defined in Definition \ref{def1} with the Fourier multiplier $D$ and the weight $\psi_{0}$ being in \eqref{psi}. Then,  there exists $0<T^*\leq T_{0}$  such that the  problem \eqref{homGP} has a unique solution  $e^{\Phi(t,D)}w\in \tilde{L}^{\infty}_{T^{*}}(B_{\psi}^{\frac{1}{2},0})$,
 where $\Phi(t,D)$  is the operator associated with the symbol $\Phi(t,\xi)$ being given in \eqref{onePhi}.
   \end{theorem}
 Moreover, we shall have the following blowup result:
\begin{theorem}\label{blowup}
For a given  $T>0$,  assume that the initial data and outer flow satisfy
\begin{equation}\label{Blowupcondition1}
u_0(0,y)=U(0,0)=0, ~U_{x}(t,0)\geq0, ~u_{0x}(0,y)\leq U_{x}(0,0)
\end{equation}
and there is $M>0$ depending on $T$ and $\|U_{x}(\cdot,0)\|_{L^{\infty}([0,T])}$ such that
\begin{equation}\label{Blowupcondition2}
\begin{split}
\int_{0}^{+\infty}\rho(y)\big(U_{x}(0,0)-u_{0x}(0,y)\big)dy >M
\end{split}
\end{equation}
for a weight function $\rho(y)$ given in \eqref{rho}, if the smooth solution $u\in C^3(Q_T)$ of \eqref{OGP}
satisfies
\begin{equation}\label{Blowupcondition3}
\lim \limits_{y \to +\infty}(u-U)e^{\psi}= 0 ~~{\rm and}~~ \lim \limits_{y \to +\infty}(u_x-U_x)e^{y} =0,
\end{equation}
 then the  $W^{1,\infty}-$norm of $u$ will blow up in a finite time, where the weight $\psi(t,y)$ is the same as given in  Theorem \ref{existence}.
 Moreover, the constant $M$ in the condition \eqref{Blowupcondition2} can be zero when the outflow velocity vanishes identically, $U(t,x)\equiv 0$.
\end{theorem}

\begin{remark}
The exponential decay property \eqref{Blowupcondition3} and  the last inequality in the condition \eqref{Blowupcondition1} are imposed to deal with the integral term given in \eqref{OGP}. In particular, when $U(t,x)\equiv0$,  the blowup result can be obtained for any nonzero initial data satisfying the simplified version of \eqref{Blowupcondition1}
\begin{equation}\label{Blowupcondition1s}
u_0(0,y)=0, ~u_{0x}(0,y)\leq 0,
\end{equation}
 which differs from the result obtained in \cite{Kuk-V-W} for the classical Prandtl equation.
\end{remark}
The rest of this paper is organised as  follows:  In  section 2,  we apply the Littlewood-Palay theory to establish  the existence and uniqueness of the solution to the problem \eqref{homGP}. In section 3,  we analyze the blowup of a smooth solution to the problem \eqref{OGP} under the conditions \eqref{Blowupcondition1}, \eqref{Blowupcondition2} and \eqref{Blowupcondition3}.

\section{Local well-posedness}

First, let us recall some basic knowledge on the Littlewood-Paley  theory and introduce the function spaces, one can refer to \cite{Z-Z, Zhu-W} for the related definitions and properties.

Let $(\varphi,\chi)$ be smooth functions such that
$$\rm{supp}~\varphi \subset \Big\{\tau \in \mathbb{R}\big|\frac{3}{4}\leq|\tau|\leq\frac{8}{3}\Big\},~~\rm{supp}~\chi \subset \Big\{\tau\in \mathbb{R}\big| |\tau|\leq\frac{4}{3}\Big\}$$
satisfying
$$\sum_{k\in\mathbb{N}} \varphi(2^{-k}\tau)=1 \quad (\forall \tau\neq0),\qquad\quad
\chi(\tau)+\sum_{k\geq0}\varphi(2^{-k}\tau)=1\quad (\forall \tau\in\mathbb{R}).$$
For any given $f\in \mathcal{S}'(\mathbb{R}_x)$, denote by $S_kf=\mathcal{F}^{-1}[\chi(2^{-k}|\xi|)\mathcal{F}[f]]$, and
$$\Delta_k f=\begin{cases}
\mathcal{F}^{-1}[\varphi(2^{-k}|\xi|)\mathcal{F}[f]],~~k\geq 0,\cr
\mathcal{F}^{-1}[\chi(|\xi|)\mathcal{F}[f]],~~~~~~~k=-1,\cr
0,~~~~~~~~~~~~~~~~~~~~~~~~~~k\leq-2,
\end{cases}
$$
where $\mathcal{F}$ and $\mathcal{F}^{-1}$ denote the Fourier transform and the inverse Fourier transform in the $x$-variable.

Introduce the following function spaces with parameters $s>0$, $l\in \mathbb{N}_{+}$ and $p\in [1,+\infty]$.
\begin{definition}\label{def1}
(i) The space $B^{s}$  is the set of functions  $u\in \mathcal{S}'(\mathbb{R})$ such that
$$\|u\|_{B^{s}}:=\sum_{k\in \mathbb{Z}}2^{ks}\|\Delta_{k}u\|_{L^2(\mathbb{R})}<+\infty.$$
(ii) The space $B_{\psi}^{s,l}$, with a positive function $\psi(y)$, is the space of functions  $u\in \mathcal{S}'(\mathbb{R}^2_{+})$ such that
$$\|u\|_{B^{s,l}_{\psi}}:=\sum_{j=0}^{l}\sum_{k\in \mathbb{Z}}2^{ks}\|e^{\psi(y)}\Delta_{k}\partial^{j}_{y}u\|_{L^2(\mathbb{R}_x\times\mathbb{R}_{+})}<+\infty.$$
(iii) The space $\tilde{L}^{p}_{t}(B^{s})$ is  defined as the
 completion of $C([0,t];\mathcal{S}(\mathbb{R}))$ with the norm
$$\|u\|_{\tilde{L}^{p}_{t}(B^s)}:=
\sum_{k\in \mathbb{Z}}2^{ks}\big(\int_{0}^{t}
\|\Delta_{k}u(t',\cdot)\|^{p}_{L^2(\mathbb{R})}dt'\big)^{\frac{1}{p}}.$$
(iv) For any positive function $\psi(t',y)$ and nonnegative $f(t')\in L_{loc}^{1}(\mathbb{R}_{+})$, the space $\tilde{L}^{p}_{t,f}(B_{\psi}^{s,l})$ is  defined as the
 completion of $C([0,t];\mathcal{S}(\mathbb{R}^{2}_{+}))$ with the norm
$$\|u\|_{\tilde{L}^{p}_{t,f}(B^{s,l}_{\psi})}
:=\sum_{j=0}^{l}\sum_{k\in \mathbb{Z}}2^{ks}\big(\int_{0}^{t}
f(t')\|e^{\psi(t',y)}\Delta_{k}\partial^{j}_{y}u(t',\cdot)\|^{p}_{L^2(\mathbb{R}_x\times\mathbb{R}_{+})}dt'\big)^{\frac{1}{p}}.$$
Denote $\tilde{L}^{p}_{t,1}(B_{\psi}^{s,l})$ by $\tilde{L}^{p}_{t}(B_{\psi}^{s,l})$ for simplicity when $f(t')\equiv 1$. The above notations can be properly changed when $p=+\infty$.
\end{definition}

\subsection{Apriori estimates}

 \quad The main step is to establish apriori estimates of solutions to the problem \eqref{homGP}, and the existence of solutions can be obtained in a usual way by constructing approximate solution sequence through linearization and a proper iteration scheme, and proving the convergence of approximate solutions via the apriori estimates. Firstly, we shall study the apriori estimates for \eqref{homGP}.

 Similar to that given in \cite{Z-Z}, to obtain energy estimates for solutions of \eqref{homGP}, we introduce the weights
\begin{equation}\label{psi}
\psi(t,y)=\frac{1+y^{2}}{16(1+t)^{\gamma}}\quad \text{and}\quad  \psi_{0}(y)=\frac{1+y^{2}}{16},
\end{equation}
with  $\gamma\geq 2$.

Denote by $\hat{w}(t,\xi,\eta)$ the Fourier transform of $w(t,x,y)$ in the $x$-variable, and
$$
 w_{\Phi}(t,x,y)=\mathcal{F}^{-1}_{\xi\rightarrow x}[e^{\Phi(t,\xi)}\hat{w}(t,\xi,y)]$$
 for a given locally bounded function $\Phi(t,\xi)$.
 To deal with the loss of derivatives in the $x$-variable in \eqref{homGP}, for any given $\delta>0$,  we shall take
 \begin{equation}\label{Phi}
\Phi^{\delta}(t,\xi)=(\delta-\lambda \theta(t))\langle\xi\rangle
\end{equation}
with $\langle \xi\rangle=1+|\xi|$ and a parameter $\lambda$, in which $\theta(t)$ is mainly used to deal with enery estimates for the nonlinear terms and is determined by the following problem
\begin{equation}\label{theta}
\left\{
\begin{array}{l}
\dot{\theta}
={\langle t\rangle} ^{\frac{\gamma}{4}}\|\partial_{y}w_{\Phi}\|_{B_{\psi}^{\frac{1}{2},0}}
+{\langle t\rangle} ^{\frac{\gamma}{4}}\|U_{\Phi}\|_{B^{\frac{1}{2}}}
+{\langle t\rangle} ^{\frac{\gamma}{2}}\|w_{\Phi}\|^2_{B_{\psi}^{1,0}}
+\|w_{\Phi}\|^2_{B_{\psi}^{\frac{1}{2},0}}\\
~~~~~+{\langle t\rangle} ^{\frac{1}{2}}\|U_{\Phi}\|^2_{B^{\frac{1}{2}}}+{\langle t\rangle} ^{\frac{1}{2}}\|U_{\Phi}\|^2_{B^{1}}+{\langle t\rangle}^{\gamma},\\
\theta\big|_{t=0}=0,
\end{array}
\right.
\end{equation}
with
\begin{align}\label{onePhi}
 \Phi(t,\xi):=\Phi^{1}(t,\xi)=(1-\lambda \theta(t))\langle\xi\rangle.
\end{align}

 If $w(t,x,y)$ is a classical solution of the problem \eqref{homGP}, then  we know that $w_{\Phi}=\mathcal{F}^{-1}_{\xi\to x}[e^{\Phi(t,\xi)}\hat{w}(t,\xi,y)]$ satisfies the following equation
\begin{equation}\label{weighHomGP}
\begin{split}
&\partial_t w_{\Phi}
+\lambda\dot{\theta}\langle D\rangle w_{\Phi}
+[(w+u^s)\partial_{x}w]_{\Phi}
+[w\partial_{x}u^s]_{\Phi}\\
&~~~~~~~+\left[(-\int_{0}^{y}\partial_{x}(w+u^s) dy')\partial_{y}(w+w^s)\right]_{\Phi}
+\left[\int_{+\infty}^{y}w dy'\right]_{\Phi}\\
=&\partial_{y}^{2}w_{\Phi}
+(1-\phi)[\partial_t U+(1+\phi)U\partial_x U]_{\Phi}
-\int_{y}^{+\infty}U_{\Phi}(1-\phi)dy'
\end{split}
\end{equation}
with $\Phi(t,\xi)$ being given in \eqref{onePhi}.

Acting the dyadic operator $\Delta_{k}$ on \eqref{weighHomGP} and taking $L^2(Q_T)$ inner product with $e^{2\psi}\Delta_{k}w_{\Phi}$ for $\psi(t,y)$ given in \eqref{psi}, it follows
\begin{align}\label{EnergyForm}
&\Big(e^{\psi}\partial_t \Delta_{k}w_{\Phi}\Big|e^{\psi}\Delta_{k}w_{\Phi}\Big)
 +\lambda\Big(\dot{\theta}\langle D\rangle e^{\psi}\Delta_{k}w_{\Phi}\Big|e^{\psi}\Delta_{k}w_{\Phi}\Big)
 -\Big(e^{\psi}\Delta_{k}\partial_{y}^{2}w_{\Phi}\Big|e^{\psi}\Delta_{k}w_{\Phi}\Big) \nonumber \\
 =&
-\Big(e^{\psi}\Delta_{k}[(w+u^s)\partial_{x}w]_{\Phi}\Big|e^{\psi}\Delta_{k}w_{\Phi}\Big)
-\Big(e^{\psi}\Delta_{k}[w\partial_{x}u^s]_{\Phi}\Big|e^{\psi}\Delta_{k}w_{\Phi}\Big) \nonumber \\
&+\Big(e^{\psi}\Delta_{k}[(\int_{0}^{y}\partial_{x}(w+u^s)
dy')\partial_{y}(w+u^s)]_{\Phi} \Big| e^{\psi}\Delta_{k}w_{\Phi}\Big)
-\Big(e^{\psi}\Delta_{k}[\int_{+\infty}^{y}w dy']_{\Phi}\Big|e^{\psi}\Delta_{k}w_{\Phi}\Big) \nonumber \\&
+\Big(e^{\psi}(1-\phi)\Delta_{k}[\partial_t U+(1+\phi)U\partial_x U]_{\Phi}\Big|e^{\psi}\Delta_{k}w_{\Phi}\Big)
 -\Big(e^{\psi}\int_{y}^{+\infty}\Delta_{k}U_{\Phi}(1-\phi)dy'\Big|e^{\psi}\Delta_{k}w_{\Phi}\Big) \nonumber \\
 :=&\sum_{i=1}^{6}J_i \nonumber,\\
\end{align}
where $(\cdot|\cdot)$ represents the inner product in $L^2(Q_T)$.

In the following calculation, for convenience we shall denote $a\leq Cb$ ($a\geq Cb$) by $a\lesssim b$ ($a\gtrsim b$), for a generic constant $C$ may change from line to line. Let us estimate each term given in \eqref{EnergyForm}.

First, the terms on the left hand side of \eqref{EnergyForm} can be estimated as follows:
\begin{align*}
(e^{\psi}\partial_t \Delta_{k}w_{\Phi}|e^{\psi}\Delta_{k}w_{\Phi})
=&\frac{1}{2}\|e^{\psi}\Delta_{k}w_{\Phi}(T,\cdot)\|^2_{L^2_+}
-\frac{1}{2}\|e^{\psi_{0}}\Delta_{k}e^{ \langle D\rangle}w_{0}\|^2_{L^2_+}
-(\psi_{t}e^{\psi}\Delta_{k}w_{\Phi}|e^{\psi}\Delta_{k}w_{\Phi}),\\
-(e^{\psi}\Delta_{k}\partial_{y}^{2}w_{\Phi}|e^{\psi}\Delta_{k}w_{\Phi})
=&(2\psi_{y}e^{\psi}\Delta_{k}\partial_{y}w_{\Phi}|e^{\psi}\Delta_{k}w_{\Phi})
+(e^{\psi}\Delta_{k}\partial_{y}w_{\Phi}|e^{\psi}\Delta_{k}\partial_{y}w_{\Phi})\\
\geq &\frac{1}{2}(e^{\psi}\Delta_{k}\partial_{y}w_{\Phi}|e^{\psi}\Delta_{k}\partial_{y}w_{\Phi})
-(2\psi^2_{y}e^{\psi}\Delta_{k}w_{\Phi}|e^{\psi}\Delta_{k}w_{\Phi})
\end{align*}
and
\begin{align*}
\lambda(\dot{\theta}\langle D\rangle e^{\psi}\Delta_{k}w_{\Phi}|e^{\psi}\Delta_{k}w_{\Phi})\gtrsim
(1+2^k) \lambda(\dot{\theta} e^{\psi}\Delta_{k}w_{\Phi}|e^{\psi}\Delta_{k}w_{\Phi}).
\end{align*}
 Thus, we get
\begin{lemma}\label{I_k}
If  we denote by
\begin{align*}
I(k)=&(e^{\psi}\partial_t \Delta_{k}w_{\Phi}|e^{\psi}\Delta_{k}w_{\Phi})
 +\lambda(\dot{\theta}\langle D\rangle e^{\psi}\Delta_{k}w_{\Phi}|e^{\psi}\Delta_{k}w_{\Phi})
 -(e^{\psi}\Delta_{k}\partial_{y}^{2}w_{\Phi}|e^{\psi}\Delta_{k}w_{\Phi})\\
 &+\frac{1}{2}\|e^{\psi_{0}}\Delta_{k}e^{\delta \langle D\rangle}w_{0}\|^2_{L^2_+},
\end{align*}
then one has
\begin{align*}
\sum_{k\in \mathbb{Z}}2^{\frac{k}{2}}\sqrt{I(k)}\gtrsim & \|w_{\Phi}\|_{\tilde{L}^{\infty}_{T}(B_{\psi}^{\frac{1}{2},0})}+
 \|\sqrt{-(\psi_t+2\psi_y^2)}w_{\Phi}\|_{\tilde{L}^{2}_{T}(B_{\psi}^{\frac{1}{2},0})}
 +\|\partial_yw_{\Phi}\|_{\tilde{L}^{2}_{T}(B_{\psi}^{\frac{1}{2},0})}\\
 &+\sqrt{\lambda}\big(\|w_{\Phi}\|_{\tilde{L}^{2}_{T,\dot{\theta}}(B_{\psi}^{\frac{1}{2},0})}
 +\|w_{\Phi}\|_{\tilde{L}^{2}_{T,\dot{\theta}}(B_{\psi}^{1,0})}
 \big).
\end{align*}
\end{lemma}
Now it remains to control the right hand side of \eqref{EnergyForm} term by term.
We shall mainly study the estimates for $J_4$ and $J_6$, and the following result for other terms
 was obtained in \cite{Zhu-W}.
\begin{lemma} (\cite[pp.~8-14]{Zhu-W})\label{J_i}
For any given $\sigma>0$, there is a constant $C_{\sigma}>0$ such that the terms $J_1$, $J_{2}$, $J_3$ and $J_5$ given in \eqref{EnergyForm} can be bounded as follows:
\begin{align*}
\sum_{k\in \mathbb{Z}}2^{\frac{k}{2}}\sqrt{|J_1(k)|}\lesssim&
C_{\sigma}\|w_{\Phi}\|_{\tilde{L}^{2}_{T,\dot{\theta}}(B_{\psi}^{1,0})}
+\sigma\|\partial_{y}w_{\Phi}\|_{\tilde{L}^{2}_{T}(B_{\psi}^{\frac{1}{2},0})}
+\sigma T^{\frac{1}{2}}\|U_{\Phi}\|_{\tilde{L}^{\infty}_{T}(B^{\frac{1}{2}})},\\
\sum_{k\in \mathbb{Z}}2^{\frac{k}{2}}\sqrt{|J_2(k)|}\lesssim&
\|w_{\Phi}\|_{\tilde{L}^{2}_{T,\dot{\theta}}(B_{\psi}^{1,0})}
+\sigma T^{\frac{1}{2}}\|U_{\Phi}\|_{\tilde{L}^{\infty}_{T}(B^{1})},\\
\sum_{k\in \mathbb{Z}}2^{\frac{k}{2}}\sqrt{|J_3(k)|}\lesssim &
C_{\sigma}\|w_{\Phi}\|_{\tilde{L}^{2}_{T,\dot{\theta}}(B_{\psi}^{1,0})}
+C_{\sigma}\|U_{\Phi}\|_{\tilde{L}^{\infty}_{T}(B^{\frac{3}{2}})}\|yw_{\Phi}\|_{\tilde{L}^{2}_{T}(B_{\psi}^{\frac{1}{2},0})}
+\sigma\|\partial_{y}w_{\Phi}\|_{\tilde{L}^{2}_{T}(B_{\psi}^{\frac{1}{2},0})}\\
&+\sigma T^{\frac{1}{2}}\|U_{\Phi}\|_{\tilde{L}^{\infty}_{T}(B^{\frac{1}{2}})}\|U_{\Phi}\|_{\tilde{L}^{\infty}_{T}(B^{\frac{3}{2}})}
+\sigma T^{\frac{1}{2}}\|U_{\Phi}\|_{\tilde{L}^{\infty}_{T}(B^{\frac{1}{2}})}
\end{align*}
and
\begin{align*}
\sum_{k\in \mathbb{Z}}2^{\frac{k}{2}}\sqrt{|J_5(k)|}\lesssim &
\|w_{\Phi}\|_{\tilde{L}^{2}_{T,\dot{\theta}}(B_{\psi}^{1,0})}
+ T^{\frac{1}{2}}\big(\|U_{\Phi}\|_{\tilde{L}^{\infty}_{T}(B^{\frac{1}{2}})}+
\|U_{\Phi}\|_{\tilde{L}^{\infty}_{T}(B^{1})}\big)
+\|w_{\Phi}\|_{\tilde{L}^{2}_{T}(B_{\psi}^{\frac{1}{2},0})}\\
&+ (T^{\frac{3}{2}}-1)^{\frac{1}{2}}\|\partial_{t}U_{\Phi}\|_{\tilde{L}^{\infty}_{T}(B^{\frac{1}{2}})}.
\end{align*}
\end{lemma}

\begin{lemma}\label{J46}
For $J_4(k)$ and $J_{6}(k)$ given in \eqref{EnergyForm}, we have
\begin{align}\label{J4}
\sum_{k\in \mathbb{Z}}2^{\frac{k}{2}}\sqrt{|J_4(k)|}\lesssim
 \|yw_{\Phi}\|_{\tilde{L}^{2}_{T}(B_{\psi}^{\frac{1}{2},0})}
+\|w_{\Phi}\|_{\tilde{L}^{2}_{T,\dot{\theta}}(B_{\psi}^{\frac{1}{2},0})}
\end{align}
and
\begin{align}\label{J6}
\sum_{k\in \mathbb{Z}}2^{\frac{k}{2}}\sqrt{|J_6(k)|}\lesssim
 ({\langle T\rangle} ^{\frac{5}{2}}-1)^{\frac{1}{2}}\|U_{\Phi}\|_{\tilde{L}^{\infty}_{T}(B_{\psi}^{\frac{1}{2},0})}
+\|w_{\Phi}\|_{\tilde{L}^{2}_{T}(B_{\psi}^{\frac{1}{2},0})}.
\end{align}
\end{lemma}
\begin{proof}
The term $J_4$ given in \eqref{EnergyForm} can be controlled as follows:
\begin{align*}
|J_4(k)|&=|(e^{\psi}\Delta_{k}[\int_{+\infty}^{y}w dy']_{\Phi}|e^{\psi}\Delta_{k}w_{\Phi})|\\
&\leq \int_{0}^{T}
\|e^{\psi}\Delta_{k}w_{\Phi}\|_{L^2_{+}}\Big(\int_{\mathbb{R}_+}dy\int_{\mathbb{R}}dx
\big (e^{2\psi}\int_{+\infty}^{y}e^{2\psi}(\Delta_{k}w _{\Phi})^2dy'
\int_{+\infty}^{y}e^{-2\psi}dy'\big)\Big)^{\frac{1}{2}} dt\\
&\lesssim \int_{0}^{T}
\langle t\rangle^{\frac{\gamma}{2}}\|e^{\psi}\Delta_{k}w_{\Phi}\|_{L^2_{+}}\Big(\int_{\mathbb{R}_+}dy\int_{\mathbb{R}}dx
\int_{+\infty}^{y}e^{2\psi}(\Delta_{k}w _{\Phi})^2dy'
\Big)^{\frac{1}{2}} dt\\
&\leq\int_{0}^{T}
\langle t\rangle^{\frac{\gamma}{2}}(\|ye^{\psi}\Delta_{k}w_{\Phi}\|_{L^2_{+}}+\|e^{\psi}\Delta_{k}w_{\Phi}\|_{L^2_{+}})
\|e^{\psi}\Delta_{k}w_{\Phi}\|_{L^2_{+}} dt  \\
&\leq\frac{1}{2}\int_{0}^{T}
\|ye^{\psi}\Delta_{k}w_{\Phi}\|^2_{L^2_{+}}dt+\int_{0}^{T} (\langle t\rangle^{\frac{\gamma}{2}}+ \frac{1}{2}\langle t\rangle^{\gamma})
\|e^{\psi}\Delta_{k}w_{\Phi}\|^2_{L^2_{+}} dt,
\end{align*}
which implies the estimate \eqref{J4} immediately by using
the definition of $\tilde{L}^{p}_{T}(B_{\psi}^{\frac{1}{2},0})$,
$\tilde{L}^{p}_{T,\dot{\theta}}(B_{\psi}^{\frac{1}{2},0})$ and $\dot{\theta}$.

For  the terms $J_6$, it yields
\begin{align*}
|J_6(k)|&=|(e^{\psi}\int_{y}^{+\infty}\Delta_{k}U_{\Phi}(1-\phi)dy'|e^{\psi}\Delta_{k}w_{\Phi})|\\
&\lesssim \int_{0}^{T}\int_{\mathbb{R}_+}e^{\psi}\int_{y}^{+\infty}(1-\phi)dy'\|\Delta_{k}U_{\Phi}\|_{L^2_x}
\|e^{\psi}\Delta_{k}w_{\Phi}\|_{L^2_x}dydt\\
&\lesssim \int_{0}^{T}\|e^{\psi}\int_{y}^{+\infty}(1-\phi)dy'\|_{L^2_{y}}\|\Delta_{k}U_{\Phi}\|_{L^2_x}
\|e^{\psi}\Delta_{k}w_{\Phi}\|_{L^2_{+}}dt\\
&\lesssim
({\langle T\rangle} ^{\frac{5}{2}}-1)^{\frac{1}{2}}\|\Delta_{k}U_{\Phi}\|_{L^{\infty}_{t}L^2_{x}}(\int_{0}^{T}
\|e^{\psi}\Delta_{k}w_{\Phi}\|^2_{L^2_{+}}
dt)^{\frac{1}{2}},
\end{align*}
where $\|e^{\psi}\int_{y}^{\infty}(1-\phi)dy'\|_{L^2_{y}}\lesssim {\langle t\rangle} ^{\frac{3}{4}} $ has been used. By using the Cauchy inequality, we conclude the estimate \eqref{J6}.
\end{proof}

Now the apriori estimates are given in the following theorem.
\begin{theorem}\label{apriori}
 Suppose that $w(t,x,y)$ is a classical solution of the problem \eqref{homGP}, then there exist $T_2>0$ and a positive constant $G$ such that there holds
\begin{equation}\label{EnergyFinal}
 \begin{split}
 &\|w_{\Phi}\|_{\tilde{L}^{\infty}_{T}(B_{\psi}^{\frac{1}{2},0})}+
 \|\sqrt{-(\psi_t+2\psi_y^2)}w_{\Phi}\|_{\tilde{L}^{2}_{T}(B_{\psi}^{\frac{1}{2},0})}
 +\|\partial_yw_{\Phi}\|_{\tilde{L}^{2}_{T}(B_{\psi}^{\frac{1}{2},0})}\\
 &~~~~+\sqrt{\lambda}\Big(\|w_{\Phi}\|_{\tilde{L}^{2}_{T,\dot{\theta}}(B_{\psi}^{\frac{1}{2},0})}
 +\|w_{\Phi}\|_{\tilde{L}^{2}_{T,\dot{\theta}}(B_{\psi}^{1,0})}
 \Big)\\
 \leq &G\Big(\|e^{\langle D\rangle}w_{0}\|_{B^{\frac{1}{2},0}_{\psi_{0}}}
 +T^{\frac{1}{2}}(\|e^{\langle D\rangle}U\|_{\tilde{L}^{\infty}_{T}(B^{\frac{1}{2}})}+
 \|e^{\langle D\rangle}U\|_{\tilde{L}^{\infty}_{T}(B^{1})})\\
 &+({\langle T\rangle} ^{\frac{3}{2}}-1)^{\frac{1}{2}}\|e^{\langle D\rangle}\partial_tU\|_{L^{\infty}_{T}(B^{\frac{1}{2}})}
+({\langle T\rangle} ^{\frac{5}{2}}-1)^{\frac{1}{2}}\|e^{\langle D\rangle}U\|_{L^{\infty}_{T}(B^{\frac{1}{2}})}\\
 &+\sigma T^{\frac{1}{2}}\|e^{\langle D\rangle}U\|_{\tilde{L}^{\infty}_{T}(B^{\frac{1}{2}})}\|e^{\langle D\rangle}U\|_{\tilde{L}^{\infty}_{T}(B^{\frac{3}{2}})}\Big)
 \end{split}
 \end{equation}
 for any  $0<T\leq T_{2}$, and the weight $\Phi(t,\xi)$ is positive in $[0,T_2]$.
\end{theorem}
\begin{proof}
 Combining Lemmas \ref{I_k}, \ref{J_i} and \ref{J46}, there is a constant $G>0$, such that
 \begin{equation}\label{EnergyIneq}
 \begin{split}
 &\|w_{\Phi}\|_{\tilde{L}^{\infty}_{T}(B_{\psi}^{\frac{1}{2},0})}+
 \|\sqrt{-(\psi_t+2\psi_y^2)}w_{\Phi}\|_{\tilde{L}^{2}_{T}(B_{\psi}^{\frac{1}{2},0})}
 +\|\partial_yw_{\Phi}\|_{\tilde{L}^{2}_{T}(B_{\psi}^{\frac{1}{2},0})}\\
 &+\sqrt{\lambda}\Big(\|w_{\Phi}\|_{\tilde{L}^{2}_{T,\dot{\theta}}(B_{\psi}^{\frac{1}{2},0})}
 +\|w_{\Phi}\|_{\tilde{L}^{2}_{T,\dot{\theta}}(B_{\psi}^{1,0})}
 \Big)\\
\leq &G\Big(\|e^{ \langle D\rangle}w_{0}\|_{B^{\frac{1}{2},0}_{\psi_{0}}} +C_{\sigma}\|w_{\Phi}\|_{\tilde{L}^{2}_{T,\dot{\theta}}(B_{\psi}^{1,0})} +\sigma\|\partial_yw_{\Phi}\|_{\tilde{L}^{2}_{T}(B_{\psi}^{\frac{1}{2},0})}
 +\|yw_{\Phi}\|_{\tilde{L}^{2}_{T}(B_{\psi}^{\frac{1}{2},0})}\\
 &
  +T^{\frac{1}{2}}(\|U_{\Phi}\|_{\tilde{L}^{\infty}_{T}(B^{\frac{1}{2}})}+
 \|U_{\Phi}\|_{\tilde{L}^{\infty}_{T}(B^{1})})
 +C_{\sigma}\|U_{\Phi}\|_{\tilde{L}^{\infty}_{T}(B^{\frac{3}{2}})}
 \|yw_{\Phi}\|_{\tilde{L}^{2}_{T}(B_{\psi}^{\frac{1}{2},0})}
 \\
 &+({\langle T\rangle} ^{\frac{3}{2}}-1)^{\frac{1}{2}}\|[\partial_tU]_{\Phi}\|_{L^{\infty}_{T}(B^{\frac{1}{2}})}+({\langle T\rangle} ^{\frac{5}{2}}-1)^{\frac{1}{2}}\|U_{\Phi}\|_{L^{\infty}_{T}(B^{\frac{1}{2}})}\\
 &+\sigma T^{\frac{1}{2}}\|U_{\Phi}\|_{\tilde{L}^{\infty}_{T}(B^{\frac{1}{2}})}
 \|U_{\Phi}\|_{\tilde{L}^{\infty}_{T}(B^{\frac{3}{2}})}
 +\|w_{\Phi}\|_{\tilde{L}^{2}_{T,\dot{\theta}}(B_{\psi}^{\frac{1}{2},0})}\Big)
 \end{split}
 \end{equation}
  for any given $\sigma>0$.

Noticing that
\begin{align}\label{weightControl}
\sqrt{-(\psi_t+2\psi_y^2)}\geq y \sqrt{\frac{\gamma-1}{16(1+t)^{\gamma+1}}},
\end{align}
by choosing $\lambda$ and $\gamma$ large enough, such that
\begin{align}\label{weightControlApp}
G\big(C_{\sigma}\|U_{\Phi}\|_{\tilde{L}^{\infty}_{T_{1}}(B^{\frac{3}{2}})}+1\big)
<\sqrt{\frac{\gamma-1}{16(1+T_{1})^{\gamma+1}}}
\end{align}
holds for a fixed $T_{1}>0$.

 Therefore, by using the inequalities \eqref{weightControl} and \eqref{weightControlApp}, the estimate \eqref{EnergyFinal} is derived from \eqref{EnergyIneq}
 for any  $0<T\leq T_{1}$.

On the other hand, in view of  \eqref{theta} and \eqref{EnergyFinal}, there exists   a constant $C(w_0,U,t)$ depending on $w_0$, $U$ and $t$ such that
 \begin{align*}
\theta(t)=\int_{0}^{t}\dot{\theta}dt'
=&\int_{0}^{t}\Big({\langle t'\rangle} ^{\frac{\gamma}{4}}\|\partial_{y}w_{\Phi}\|_{B_{\psi}^{\frac{1}{2},0}}
+{\langle t'\rangle} ^{\frac{\gamma}{4}}\|U_{\Phi}\|_{B^{\frac{1}{2}}}\\
&+{\langle t'\rangle} ^{\frac{\gamma}{2}}\|w_{\Phi}\|^2_{B_{\psi}^{1,0}}
+\|w_{\Phi}\|^2_{B_{\psi}^{\frac{1}{2},0}}+
{\langle t'\rangle} ^{\frac{1}{2}}\|U_{\Phi}\|^2_{B^{\frac{1}{2}}}+{\langle t'\rangle} ^{\frac{1}{2}}\|U_{\Phi}\|^2_{B^{1}}+{\langle t'\rangle}^{\gamma}\Big)dt'\\
\lesssim & ({\langle t\rangle} ^{\frac{\gamma}{2}+1}-1)^{\frac{1}{2}}
\|\partial_yw_{\Phi}\|_{\tilde{L}^{2}_{t}(B_{\psi}^{\frac{1}{2},0})}
+({\langle t\rangle} ^{\frac{\gamma}{4}+1}-1)
\|e^{\langle D\rangle}U\|_{\tilde{L}^{\infty}_{t}(B^{\frac{1}{2}})}
+{\langle t\rangle} ^{\frac{\gamma}{2}}
\|w_{\Phi}\|_{\tilde{L}^{2}_{t}(B_{\psi}^{1,0})}\\
&+\|w_{\Phi}\|_{\tilde{L}^{2}_{t}(B_{\psi}^{\frac{1}{2},0})}
+({\langle t\rangle} ^{\frac{3}{2}}-1)
(\|e^{\langle D\rangle}U\|^2_{\tilde{L}^{\infty}_{t}(B^{\frac{1}{2}})}
+\|e^{\langle D\rangle}U\|^2_{\tilde{L}^{\infty}_{t}(B^{1})})
+({\langle t\rangle} ^{\gamma+1}-1)\\
\lesssim &C(w_0,U,t).
\end{align*}
Therefore one can choose $0<T_2\leq T_1$ properly small such that
 \begin{equation}\label{chooseT}
 0<T_{2}\leq \sup\limits_{t>0}\left\{\theta(t)<\frac{1}{\lambda}\right\},
 \end{equation}
which guarantees the weight $\Phi(t,\xi)$ defined in \eqref{Phi} is positive on $[0,T_2]$. Thereby we obtain the apriori estimate \eqref{EnergyFinal} for $0< T\leq T_2$.
\end{proof}
\begin{remark}\label{deltaEst}
If the weight $\Phi$ is replaced by $\Phi^{\delta}$, there exists a time $T_{\delta}$ such that the apriori estimate in Theorem \ref{apriori} is still valid in $[0,T_{\delta}]$, with $e^{\langle D\rangle}$ being replaced by $e^{\delta\langle D\rangle}$.
\end{remark}

\subsection{Existence of a solution}

To obtain the existence of a solution to the problem \eqref{homGP}, similar to that given in \cite{Zhu-W},  consider the approximation of \eqref{homGP} as follows for any integer $n\geq 1$,
 \begin{equation}\label{approximate}
\left\{
\begin{array}{l}
\partial_t w_{n} +(w_{n}+u^s)\partial_{x}w_{n}+w_{n}\partial_{x}u^s-\int_{0}^{y}\partial_{x}(w_{n}+u^s)dy'
\partial_{y}(w_{n}+u^s)\\ [8pt]
~~~~~~~~~~~+\int_{+\infty}^{y}w_{n} dy'
-\partial_{y}^{2}w_{n}-\frac{1}{n^2}\partial_{x}^2w_{n}\\ [8pt]
~~~~~~~=(1-\phi)(\partial_t U+(1+\phi)U\partial_x U)-\int_{y}^{+\infty}U(1-\phi)dy',\\ [8pt]
w_{n}\big|_{y=0}=0,~~\lim\limits_{y\to +\infty}w_{n}=0,~~w_{n}\big|_{t=0}=w_0(x,y).
\end{array}
\right.
\end{equation}
The well-posedness of the problems \eqref{approximate} can be  obtained from the classical theory of the parabolic equations, and $w_n(n\geq1)$ satisfies the same apriori estimate \eqref{EnergyFinal} on $[0,T_2]$ as given in Theorem \ref{apriori}.

For any fixed $\delta\in (0,1)$ in $\Phi^{\delta}$ given in \eqref{Phi},
by using
$$\|(w_n)_{\Phi^{\delta}}\|_{\tilde{L}^{2}_{T}(B_{\psi}^{2,0})}\lesssim
\|(w_n)_{{\Phi}^{1}}\|_{\tilde{L}^{2}_{T}(B_{\psi}^{1,0})}\quad (n\geq 1),$$
in a way similar to the proof of uniqueness given in the next subsection, we can get that there exists $0<T^{*}\leq T_{2}$ such that $V=w_{n+1}-w_{n}$ satisfies the following estimate,
$$\|V_{\Phi^{\delta}}\|_{\tilde{L}^{\infty}_{T}(B_{\psi}^{\frac{1}{2},0})}
+\|\partial_yV_{\Phi^{\delta}}\|_{\tilde{L}^{2}_{T}(B_{\psi}^{\frac{1}{2},0})}
 \lesssim (\frac{1}{n^2}+\frac{1}{(n+1)^{2}})\mathcal{R},\quad \text{for}~0<T\leq T^{*}$$
where $\mathcal{R}$ represents the right hand side of \eqref{EnergyFinal}. Therefore for any fixed $0<\delta\leq 1$ in the weight $\Phi^{\delta}$, $\{(w_n)_{\Phi^{\delta}}\}_{n\geq 1}$ is a Cauchy sequence in $\tilde{L}^{\infty}_{T}(B_{\psi}^{\frac{1}{2},0})$, which  follows the existence of a solution to the problem \eqref{homGP}.

\subsection{Uniqueness of the solution}
In this subsection, we study the uniqueness of the solution to the problem \eqref{homGP}.
Suppose that the problem \eqref{homGP} has two solutions $w^1$ and $w^2$, obviously $V=w^1-w^2$ satisfies
the following problem,
\begin{equation}\label{errGP}
\left\{
\begin{array}{l}
\partial_t V +\int_{+\infty}^{y}V dy'-\partial_{y}^{2}V
=-(w^2+u^s)\partial_{x}V -V\partial_{x}w^1+V\partial_{x}u^s
\\ [8pt] ~~~~~~~~~~~~~~~~~~~~~~~~~~~~~~~~~+
\int_{0}^{y}\partial_{x}(w^2+u^s)
\partial_{y}V
+\int_{0}^{y}\partial_x Vdy'\partial_y(w^1+u^s),\\ [8pt]
V\big|_{y=0}=0,~~\lim\limits_{y\to +\infty}V=0, ~~V\big|_{t=0}=0.
\end{array}
\right.
\end{equation}
Denote the corresponding weights by $\theta^1(t),\Phi^{1}_1(t,\xi)$ and $\theta^2(t),\Phi^{1}_2(t,\xi)$ given in \eqref{theta} and \eqref{Phi} with respect to $w^1$ and $w^2$ for $\delta=1$,
we introduce  $$\Theta=\theta^1+\theta^2~\text{and}~\widehat{\Phi}^{\delta}=(\delta-\lambda\Theta(t))\langle\xi\rangle,
~\text{for any given}~\delta\in(0,1).$$


From \eqref{errGP}, one has  that
\begin{equation*}\label{weightErrGP}
\begin{split}
\partial_t V_{\hat{\Phi}^\delta} +\lambda\dot{\Theta}V_{\hat{\Phi}^\delta}&+\int_{+\infty}^{y}V_{\hat{\Phi}^\delta} dy'-\partial_{y}^{2}V_{\hat{\Phi}^\delta}
=-[(w^2+u^s)\partial_{x}V]_{\hat{\Phi}^\delta} -[V\partial_{x}w^1]_{\hat{\Phi}^\delta}
-[V\partial_{x}u^s]_{\hat{\Phi}^\delta}\\
&+[\int_{0}^{y}\partial_{x}(w^2+u^s)dy'\partial_{y}V]_{\hat{\Phi}^\delta}
+[\int_{0}^{y}\partial_x Vdy'\partial_y(w^1+u^s)]_{\hat{\Phi}^\delta}.
\end{split}
\end{equation*}

By acting the dyadic operator $\Delta_{k}$ on the above equation and taking $L^2(Q_T)$ inner product with $e^{2\psi}\Delta_{k}V_{\hat{\Phi}^\delta}$, it yields that
\begin{equation}\label{errEnergyForm}
\begin{split}
 &\Big(e^{\psi}\partial_t \Delta_{k}V_{\hat{\Phi}^\delta}\Big|e^{\psi}\Delta_{k}V_{\hat{\Phi}^\delta}\Big)
 +\lambda\Big(\dot{\Theta}\langle D\rangle e^{\psi}\Delta_{k}V_{\hat{\Phi}^\delta}\Big|e^{\psi}\Delta_{k}V_{\hat{\Phi}^\delta}\Big)
 -\Big(e^{\psi}\Delta_{k}\partial_{y}^{2}V_{\hat{\Phi}^\delta}\Big|e^{\psi}\Delta_{k}V_{\hat{\Phi}^\delta}\Big) \\
 =&
-\Big(e^{\psi}\Delta_{k}[(w^2+u^s)\partial_{x}V]_{\hat{\Phi}^\delta}\Big|e^{\psi}\Delta_{k}V_{\hat{\Phi}^\delta}\Big)
-\Big(e^{\psi}\Delta_{k}[V\partial_{x}u^s]_{\hat{\Phi}^\delta}\Big|e^{\psi}\Delta_{k}V_{\hat{\Phi}^\delta}\Big) \\
&-\Big(e^{\psi}\Delta_{k}[V\partial_{x}w^1]_{\hat{\Phi}^\delta}\Big|e^{\psi}\Delta_{k}V_{\hat{\Phi}^\delta}\Big)
+\Big(e^{\psi}\Delta_{k}[(\int_{0}^{y}\partial_{x}(w^2+u^s) dy')\partial_{y}V]_{\hat{\Phi}^\delta}\Big|e^{\psi}\Delta_{k}V_{\hat{\Phi}^\delta}\Big) \\
&+\Big(e^{\psi}\Delta_{k}[(\int_{0}^{y}\partial_{x}V dy')\partial_{y}(w^1+u^s)]_{\hat{\Phi}^\delta}\Big|e^{\psi}\Delta_{k}V_{\hat{\Phi}^\delta}\Big)
 -\Big(e^{\psi}\Delta_{k}[\int_{+\infty}^{y}V dy']_{\hat{\Phi}^\delta}\Big|e^{\psi}\Delta_{k}V_{\hat{\Phi}^\delta}\Big) \\
 :=&\sum_{i=1}^{6}J_i.
 \end{split}
\end{equation}

In a way similar to that given in Lemmas \ref{I_k} and \ref{J46},
one has the following result:
\begin{lemma}\label{eI_k}
Set
$$I(k)\triangleq \Big(e^{\psi}\partial_t \Delta_{k}V_{\hat{\Phi}^\delta}\Big|e^{\psi}\Delta_{k}V_{\hat{\Phi}^\delta}\Big)
 +\lambda\Big(\dot{\Theta}\langle D\rangle e^{\psi}\Delta_{k}V_{\hat{\Phi}^\delta}\Big|e^{\psi}\Delta_{k}V_{\hat{\Phi}^\delta}\Big)
 -\Big(e^{\psi}\Delta_{k}\partial_{y}^{2}V_{\hat{\Phi}^\delta}\Big|e^{\psi}\Delta_{k}V_{\hat{\Phi}^\delta}\Big).$$
For $J_{6}$ and $I(k)$ given at above, there hold
\begin{align*}\label{J4}
\sum_{k\in \mathbb{Z}}2^{\frac{k}{2}}\sqrt{|J_6(k)|}\lesssim
 \|yV_{\Hat{\Phi}}\|_{\tilde{L}^{2}_{T}(B_{\psi}^{\frac{1}{2},0})}
+\|V_{\Hat{\Phi}}\|_{\tilde{L}^{2}_{T,\dot{\Theta}}(B_{\psi}^{\frac{1}{2},0})}
\end{align*}
and
\begin{align*}
\sum_{k\in \mathbb{Z}}2^{\frac{k}{2}}\sqrt{I(k)}\gtrsim & \|V_{\Hat{\Phi}}\|_{\tilde{L}^{\infty}_{T}(B_{\psi}^{\frac{1}{2},0})}+
 \|\sqrt{-(\psi_t+2\psi_y^2)}V_{\Hat{\Phi}}\|_{\tilde{L}^{2}_{T}(B_{\psi}^{\frac{1}{2},0})}
 +\|\partial_yV_{\Hat{\Phi}}\|_{\tilde{L}^{2}_{T}(B_{\psi}^{\frac{1}{2},0})}\\
 &+\sqrt{\lambda}\big(\|V_{\Hat{\Phi}}\|_{\tilde{L}^{2}_{T,\dot{\Theta}}(B_{\psi}^{\frac{1}{2},0})}
 +\|V_{\Hat{\Phi}}\|_{\tilde{L}^{2}_{T,\dot{\Theta}}(B_{\psi}^{\frac{1}{2},0})}
 \big).
\end{align*}
\end{lemma}
 As shown in \cite{Z-Z,Zhu-W}, for $0<\delta<1$, there hold $$\|w^{i}_{\Hat{\Phi}^{\delta}}\|_{\tilde{L}^{\infty}_{T}(B_{\psi}^{\frac{3}{2},0})}\lesssim
\|w^{i}_{{\Phi}^{1}_{i}}\|_{\tilde{L}^{\infty}_{T}(B_{\psi}^{\frac{1}{2},0})},
~\|w^{i}_{\Hat{\Phi}^{\delta}}\|_{\tilde{L}^{2}_{T}(B_{\psi}^{\frac{3}{2},0})}\lesssim
\|w^{i}_{{\Phi}^{1}_{i}}\|_{\tilde{L}^{2}_{T}(B_{\psi}^{\frac{1}{2},0})}\quad (i=1,2).$$
 The remaining terms given on the right hand side of \eqref{errEnergyForm} can be controlled as  given in \cite[pp.~34-38]{Zhu-W}, and one concludes:
\begin{lemma}\label{eJ_i}
For any $\sigma>0$,  the terms $J_{i}~(1\leq i \leq 5)$ given in \eqref{errEnergyForm} satisfy the following estimate:
\begin{align*}
&\sum_{i=0}^{5}\sum_{k\in \mathbb{Z}}2^{\frac{k}{2}}\sqrt{|J_i(k)|}\\
\lesssim &
\big(\|\partial_{y}w^2_{{\Phi^{1}_2}}\|_{\tilde{L}^{2}_{T}(B_{\psi}^{\frac{1}{2},0})}
+T^{\frac{1}{2}}\|e^{\langle D\rangle}U\|_{\tilde{L}^{\infty}_{T}(B^{\frac{1}{2}})}\big)^{\frac{1}{2}}
\|V_{\hat{\Phi}^\delta}\|_{\tilde{L}^{\infty}_{T}(B_{\psi}^{\frac{1}{2},0})}(\langle T\rangle^{\frac{\gamma}{2}+1}-1)^{\frac{1}{4}}\\
&+T^{\frac{1}{2}}\|e^{\langle D\rangle}U\|^{\frac{1}{2}}_{\tilde{L}^{\infty}_{T}(B^{\frac{1}{2}})}
\|V_{\hat{\Phi}^\delta}\|_{\tilde{L}^{\infty}_{T}(B_{\psi}^{\frac{1}{2},0})}
+\sigma(\langle T\rangle^{\frac{\gamma}{2}+1}-1)^{\frac{1}{2}}
\|w^1_{\Phi^{1}_1}\|_{\tilde{L}^{\infty}_{T}(B_{\psi}^{\frac{1}{2},0})}
\|V_{\hat{\Phi}^\delta}\|_{\tilde{L}^{\infty}_{T}(B_{\psi}^{\frac{1}{2},0})}\\
&+\sigma\|\partial_{y}V_{\hat{\Phi}^\delta}\|_{\tilde{L}^{2}_{T}(B_{\psi}^{\frac{1}{2},0})}
+C_{\sigma}(\langle T\rangle^{\frac{\gamma}{2}+1}-1)^{\frac{1}{2}}
\|w^2_{\Phi^{1}_2}\|_{\tilde{L}^{\infty}_{T}(B_{\psi}^{\frac{1}{2},0})}
\|V_{\hat{\Phi}^\delta}\|_{\tilde{L}^{\infty}_{T}(B_{\psi}^{\frac{1}{2},0})}\\
&+C_{\sigma}\|U_{\hat{\Phi}^\delta}\|_{\tilde{L}^{\infty}_{T}(B_{\psi}^{\frac{3}{2}})}
\|yV_{\hat{\Phi}^\delta}\|_{\tilde{L}^{2}_{T}(B_{\psi}^{\frac{1}{2},0})}
+C_{\sigma}\|V_{\hat{\Phi}^\delta}\|_{\tilde{L}^{2}_{T,\dot{\Theta}}(B_{\psi}^{1,0})}
.
\end{align*}
\end{lemma}

Based on the above lemmas, we give the proof of the uniqueness part of Theorem \ref{existence} as follows.

\textbf{Proof of the uniqueness part of Theorem \ref{existence}.}
Combining Lemmas \ref{eI_k} and  \ref{eJ_i}, from \eqref{errEnergyForm} we obtain that
\begin{align*}
&\|V_{\Hat{\Phi}^\delta}\|_{\tilde{L}^{\infty}_{T}(B_{\psi}^{\frac{1}{2},0})}+
 \|\sqrt{-(\psi_t+2\psi_y^2)}V_{\Hat{\Phi}^\delta}\|_{\tilde{L}^{2}_{T}(B_{\psi}^{\frac{1}{2},0})}
 +\|\partial_yV_{\Hat{\Phi}^\delta}\|_{\tilde{L}^{2}_{T}(B_{\psi}^{\frac{1}{2},0})}\\
 &~~~~+\sqrt{\lambda}\big(\|V_{\Hat{\Phi}^\delta}\|_{\tilde{L}^{2}_{T,\dot{\Theta}}(B_{\psi}^{\frac{1}{2},0})}
 +\|V_{\Hat{\Phi}^\delta}\|_{\tilde{L}^{2}_{T,\dot{\Theta}}(B_{\psi}^{\frac{1}{2},0})}
 \big)\\
 \lesssim &
 C_{\sigma}\|V_{\hat{\Phi}^\delta}\|_{\tilde{L}^{2}_{T,\dot{\Theta}}(B_{\psi}^{1,0})}
+\big(\|\partial_{y}w^2_{{\Phi^{1}_2}}\|_{\tilde{L}^{2}_{T}(B_{\psi}^{\frac{1}{2},0})}
+T^{\frac{1}{2}}\|e^{\langle D\rangle}U\|_{\tilde{L}^{\infty}_{T}(B^{\frac{1}{2}})}\big)^{\frac{1}{2}}
\|V_{\hat{\Phi}^\delta}\|_{\tilde{L}^{\infty}_{T}(B_{\psi}^{\frac{1}{2},0})}(\langle T\rangle^{\frac{\gamma}{2}+1}-1)^{\frac{1}{4}}\\
&+T^{\frac{1}{2}}\|e^{\langle D\rangle}U\|^{\frac{1}{2}}_{\tilde{L}^{\infty}_{T}(B^{\frac{1}{2}})}
\|V_{\hat{\Phi}^\delta}\|_{\tilde{L}^{\infty}_{T}(B_{\psi}^{\frac{1}{2},0})}
+\sigma(\langle T\rangle^{\frac{\gamma}{2}+1}-1)^{\frac{1}{2}}
\|w^1_{\Phi^{1}_1}\|_{\tilde{L}^{\infty}_{T}(B_{\psi}^{\frac{1}{2},0})}
\|V_{\hat{\Phi}^\delta}\|_{\tilde{L}^{\infty}_{T}(B_{\psi}^{\frac{1}{2},0})}\\
&+\sigma\|\partial_{y}V_{\hat{\Phi}^\delta}\|_{\tilde{L}^{2}_{T}(B_{\psi}^{\frac{1}{2},0})}
+C_{\sigma}(\langle T\rangle^{\frac{\gamma}{2}+1}-1)^{\frac{1}{2}}
\|w^2_{\Phi^{1}_2}\|_{\tilde{L}^{\infty}_{T}(B_{\psi}^{\frac{1}{2},0})}
\|V_{\hat{\Phi}^\delta}\|_{\tilde{L}^{\infty}_{T}(B_{\psi}^{\frac{1}{2},0})}\\
&+C_{\sigma}\|U_{\hat{\Phi}^\delta}\|_{\tilde{L}^{\infty}_{T}(B_{\psi}^{\frac{3}{2}})}
\|yV_{\hat{\Phi}^\delta}\|_{\tilde{L}^{2}_{T}(B_{\psi}^{\frac{1}{2},0})}
+ \|yV_{\Hat{\Phi}^\delta}\|_{\tilde{L}^{2}_{T}(B_{\psi}^{\frac{1}{2},0})}
+\|V_{\Hat{\Phi}^\delta}\|_{\tilde{L}^{2}_{T,\dot{\Theta}}(B_{\psi}^{\frac{1}{2},0})}.
\end{align*}
By taking $\lambda$ and $\gamma$ large  and $T>0$ being small properly, the above inequality implies
$$\|V_{\Hat{\Phi}^\delta}\|_{\tilde{L}^{\infty}_{T}(B_{\psi}^{\frac{1}{2},0})}=0.$$

Thus, we get $V\equiv0$ in $0\leq t \leq T$, this uniqueness can be extended to the whole time interval of existence of the solution given in Section 2.1 with the aid of continuation argument. \qed

\section{Blowup of the solution}
In this section, we are interested in whether the smooth solution of the problem \eqref{OGP} exists globally in time. Set $H_{t}=\{t'\in (0,t), y> 0\}$.
Under the assumption that the initial data $u_{0}$ and outer flow $U$ satisfy the condition \eqref{Blowupcondition1},
we shall prove that the norm $\|\partial_{x}u(t',0,y)\|_{L^{\infty}(H_t)}$ of  solution
 to the problem \eqref{OGP} will blow up in $(0,T)$. This shall be obtained by developing the idea from \cite{Kuk-V-W} and a contradiction argument.

 Denote by
 $$\bar{f}(t,y)=f(t,0,y) ~\text{and}~ \bar{g}(t)=g(t,0)$$
 for functions $f(t,x,y)$ and $(g(t,x)$.

 By restricting the problem \eqref{OGP} on the plane $\{x=0\}$,
  we get that  $\bar{u}(t,y)=u(t,0,y)$ satisfies the following problem in $H_{t}$,
\begin{equation}\label{OGP_Rx} \left\{\!\!
\begin{array}{lc}
\partial_t \bar{u}+\bar{u}\overline{\partial_x u}+\bar{v}\partial_y \bar{u}+\int^y_{\infty}\bar{u}dy'-\partial_{y}^2\bar{u}=0, &\\[8pt]
\bar{u}|_{t=0}=0, &\\[8pt]
(\bar{u}, \bar{v})|_{y=0}=(0, 0),  ~~\lim \limits_{y \to +\infty }\bar{u}= 0
\end{array}
\right.
\end{equation}
where $\overline{\partial_x u}+\partial_{y}\bar{v}=0$.

Assume that $\partial_{x}u(t,0,y$ does not blow up in $H_T$, and there is a constant $M_T>0$ such that
\begin{equation}\label{suppose0}
\|\partial_{x}u(t,0,y)\|_{L^{\infty}(H_T)}\leq M_T.
\end{equation}
Owing to the assumption \eqref{suppose0}, one has the following lemma.
\begin{lemma}\label{zeroSol}
Assume that $\overline{\partial_x u}\in C^2(H_T)$  satisfies the assumption \eqref{suppose0}. Then the problem \eqref{OGP_Rx} only has the trivial solution, $\partial_x u(t, 0, y)\equiv 0$ in $H_T$, in the class that satisfies the first decay condition given in \eqref{Blowupcondition3}.
\end{lemma}
\begin{proof}
 By multiplying the equation in \eqref{OGP_Rx} by $\bar{u}e^{2\psi}$ ($\psi$ is the same as given in \eqref{psi} with the parameter $\gamma$ large) and integrating in $y$, one obtains
\begin{align}\label{uniLin}
&\frac{d}{2dt}\int_{0}^{+\infty}{\bar{u}}^2e^{2\psi}dy
-\int_{0}^{+\infty}\partial_t \psi {\bar{u}}^2e^{2\psi}dy
+\int_{0}^{+\infty}(\partial_{y} {\bar{u}})^2 e^{2\psi}dy\nonumber\\
=&-\int_{0}^{+\infty}2\psi_{y}\partial_{y}{\bar{u}} {\bar{u}} e^{2\psi}dy
+\int_{0}^{+\infty}\overline{\partial_x u}{\bar{u}}^2e^{2\psi}dy
+\int_{0}^{+\infty}\bar{v}\partial_y {\bar{u}} {\bar{u}} e^{2\psi}dy\\
&+\int_{0}^{+\infty}\int^y_{\infty}\bar{u}dy' {\bar{u}} e^{2\psi}dy
:=I_1+I_2+I_3+I_4.\nonumber
\end{align}
By using integration by parts and Young's inequality,  one can control the terms on the right hand side of \eqref{uniLin} as follows:
\begin{align*}
|I_1|\leq &\int_{0}^{+\infty}2\psi_{y}^2 {\bar{u}}^2 e^{2\psi}dy
+\frac{1}{2}\int_{0}^{+\infty}(\partial_{y}{\bar{u}})^2 e^{2\psi}dy,\\
|I_3|\leq &\frac{1}{4}\int_{0}^{+\infty} (\partial_y {\bar{u}})^{2} e^{2\psi}dy
+\int_{0}^{+\infty}\bar{v}^2 {\bar{u}}^2 e^{2\psi}dy\\
\leq &\frac{1}{4}\int_{0}^{+\infty} (\partial_y {\bar{u}})^{2} e^{2\psi}dy
+\|\overline{\partial_x u}\|^2_{L^{\infty}(H_{T})}\int_{0}^{+\infty}y^2 {\bar{u}}^2 e^{2\psi}dy
\end{align*}
and
\begin{align*}
|I_4|\leq C_T\int_{0}^{+\infty}(\int_y^{\infty}{\bar{u}}^2 e^{2\psi}dy')^{\frac{1}{2}} |{\bar{u}}| e^{\psi}dy
 \leq  \delta \int_{0}^{+\infty}y^2 {\bar{u}}^2 e^{2\psi}dy+\frac{C_T^2}{2\delta}\int_{0}^{+\infty}{\bar{u}}^2 e^{\psi}dy
\end{align*}
with $\delta$ being sufficiently small.
Thus, we conclude that
\begin{align*}\label{uniLin}
&\frac{d}{2dt}\int_{0}^{+\infty}{\bar{u}}^2e^{2\psi}dy
-\int_{0}^{+\infty}(\partial_t \psi+2\psi_y^2 ){\bar{u}}^2e^{2\psi}dy
+\frac{1}{4}\int_{0}^{+\infty}(\partial_{y} {\bar{u}})^2 e^{2\psi}dy\nonumber\\
\leq&(\|\overline{\partial_x u}\|_{L^{\infty}(H_{T})}+\frac{C_T^2}{2\delta})
\int_{0}^{+\infty}{\bar{u}}^2 e^{2\psi}dy
+(\|\overline{\partial_x u}\|^2_{L^{\infty}(H_{T})}+\delta)
\int_{0}^{+\infty}y^2 {\bar{u}}^2 e^{2\psi}dy.
\end{align*}
In view of \eqref{Blowupcondition3}, \eqref{weightControl} and \eqref{suppose0}, letting $\gamma$ be large enough in the definition of $\psi(t,y)$,  we can achieve $\bar{u}(t,y) \equiv 0~(0\leq t \leq T_1)$ by using Gronwall's inequality in the above estimate for a small time $T_1$. By a continuation argument, we get $\bar{u}(t,y) \equiv 0$ for all $0\leq t \leq T$.
\end{proof}
 Denote by $\tilde{u}(t,y)=-\partial_{x}u(t,0,y)$ and $\tilde{U}(t)=-\partial_{x}U(t,0)$.
  With the aid of  the condition \eqref{Blowupcondition1} and Lemma \ref{zeroSol}, we know from \eqref{OGP} that $w=\tilde{u}-\tilde{U}$
 satisfies the problem
\begin{equation}\label{finBGP_Rx} \left\{\!\!
\begin{array}{lc}
\partial_t w-w^2+\partial_{y}^{-1}(w+\tilde{U})\partial_yw -2\tilde{U}w
+\int^y_{\infty}w dy'-\partial_{y}^2w=0, &\\[8pt]
w|_{t=0}=w_0\triangleq \tilde{u}_0(y)-\tilde{U}(0), &\\[8pt]
w|_{y=0}=-\tilde{U}(t),  ~~\lim \limits_{y \to +\infty} w= 0,
\end{array}
\right.
\end{equation}
where $\partial_{y}^{-1}f(y):=\int_{0}^{y} f(y')dy'$.

For the problem \eqref{finBGP_Rx}, first we have the following non-negative property of the solution.
\begin{lemma}
Under the assumptions \eqref{Blowupcondition1}, \eqref{Blowupcondition3} and \eqref{suppose0}, any classical solution of the problem \eqref{finBGP_Rx} is non-negative.
\end{lemma}
\begin{proof}
Set $V=we^{y-\lambda t}$  for $\lambda>0$, then $V$ satisfies
\begin{equation}\label{modPoofPositive} \left\{\!\!
\begin{array}{lc}
\partial_t V+\big(2+\partial_{y}^{-1}(V e^{\lambda t-y}+\tilde{U})\big)\partial_y V
-\partial_{y}^{-1}(Ve^{\lambda t-y}+\tilde{U})V
 &\\[8pt]
~~~~~
+(\lambda -2\tilde{U}-1)V-\partial_{y}^2 V
=V^2e^{\lambda t -y}+e^{y}\int^{\infty}_y e^{-y}V dy', &\\[8pt]
V|_{t=0}=e^{y}{w}_{0}, &\\[8pt]
V|_{y=0}=-e^{-\lambda t}\tilde{U},  ~~\lim \limits_{y \to +\infty} V=0.
\end{array}
\right.
\end{equation}

For any fixed $\ve> 0$, we consider $V^{\ve}=V+\ve$,  which  satisfies
\begin{equation}\label{veModPoofPositive} \left\{\!\!
\begin{array}{lc}
\partial_t {V^{\ve}}+\big(2+\partial_{y}^{-1}((V^{\ve}-\ve)e^{\lambda t-y}+\tilde{U})\big)\partial_y {V^{\ve}}
-\partial_{y}^{-1}((V^{\ve}-\ve)e^{\lambda t-y}+\tilde{U}){V^{\ve}}&\\[8pt]
~~+(\lambda -2\tilde{U}-1)V^{\ve}
-\partial_{y}^2 {V^{\ve}}
=-\ve\partial_{y}^{-1}((V^{\ve}-\ve)e^{\lambda t-y})
-\ve\partial_{y}^{-1}\tilde{U}&\\[8pt]
~~+\ve(\lambda -2\tilde{U}-2)
+(V^{\ve})^2e^{\lambda t -y}-2\ve V^{\ve}e^{\lambda t-y}
+\ve^{2}e^{\lambda t-y}
+e^{y}\int^{\infty}_y e^{-y}{V^{\ve}} dy', &\\[8pt]
{V^{\ve}}|_{t=0}=e^{y}{w}_{0}+\ve, &\\[8pt]
{V^{\ve}}|_{y=0}=-e^{-\lambda t}\tilde{U}+\ve,  ~~\lim \limits_{y \to +\infty}{V^{\ve}} =\ve.
\end{array}
\right.
\end{equation}
Due to $V^{\ve}\geq \ve>0$ at $t=0$, we claim  that $V^{\ve}\geq0$ in $H_{T}$. Otherwise, let $t^{*}\in (0, T]$ be the first time such that $V^{\ve}=0$ at an interior point $(t^{*},y^{*})$,  then one has

(i) $V^{\ve} \geq 0 $ in $H_{t^{*}}$,

(ii) $V^{\ve} $ attains its minimum in $H_{t^{*}}$ at the point $(t^{*},y^{*})$.

In addition, under the assumption \eqref{suppose0} and the second condition in \eqref{Blowupcondition3}, one has that
\begin{equation}\label{suppose}
\|we^y\|_{L^{\infty}(H_T)}\leq \bar{M}_T,
\end{equation}
for a positive constant $\bar{M}_T>0$,
then  there holds
$$\ve\partial_{y}^{-1}((V^{\ve}-\ve)e^{\lambda t-y})\leq \ve \bar{M}_{T}.$$

 By noting that at $(t^{*}$, $y^{*})$, $\partial_t V^{\ve}\leq 0$, $\partial^2_y V^{\ve}\geq 0$, $V^{\ve}= 0$ and $\partial_y V^{\ve}=0$, and plugging these information into the first equation in \eqref{veModPoofPositive}, it leads to a  contradiction at $(t^{*}$, $y^{*})$ by choosing $\lambda>0$ properly large.
  As a consequence, it deduces that $V^{\ve}\geq0$ in $H_T$.
 Moreover, in virtue of the arbitrariness of $\ve$,  we conclude  that
 $$V\geq 0 ~\text{in}~H_{T},$$
 which implies
$$ w \geq 0 ~\text{in}~H_{T}.$$
Thereby we complete the proof of this lemma.
\end{proof}
 Next, we shall prove  that under certain condition, the solution $w$ of \eqref{finBGP_Rx} will tend to infinity in a finite time by constructing a Lyapunov functional.

 For this,  we define the Lyapunov functional as follows:
$$ G(t)=\int_{0}^{\infty}\rho(y) w(t,y)dy.$$
Here the nonnegative weight $\rho(y)\in W^{2,\infty}(\mathbb{R}_+)\cap C^1(\mathbb{R}_+)\cap L^1(\mathbb{R}_+)$ is given by
\begin{equation}\label{rho}
\rho(y)=
\begin{cases}
f(y),& 0\leq y\leq B,\\
g(y),& y>B,
\end{cases}
\end{equation}
with real numbers $0<A<M<B<+\infty$, where $f(y)\in C^1([0,B])\cap C^2([0,B] \setminus \{A\})$ satisfies
 \begin{enumerate}[(F1)]
 \item $f(0)=0,~f(y)>0$ for any $y\in(0,B]$,\label{f1}
 \item $yf'(y)\leq C_{f}f(y)~\text{for some}~ C_f>0 ~\text{on}~[0,B]$,\label{f2}
 \item $\int_{0}^{y} f(y') dy'+f''(y)\geq 0$, on $[0,B]\setminus \{A\}$, \label{f3}
 \item $f''(y)\leq 0,~\text{on}~ [0,B]\setminus \{A\}$,\label{f4}
 \end{enumerate}
 and the extension $g(y)\in C^2([M,+\infty)\cap L^1([M,+\infty))$ satisfies
 \begin{enumerate}[(G1)]
 \item $\lim \limits_{y\rightarrow +\infty}g(y)
     =\lim \limits_{y\rightarrow +\infty}g'(y)=0,~g(y)>0~\text{for any}~y>M$,\label{g1}
 \item $g'(y)<0~\text{and}~g''(y)>0~\text{for any}~y>M$,\label{g2}
 \item $\frac{(g')^2}{gg''}\leq \beta <1~ \text{for any}~y\geq B$.\label{g3}
\end{enumerate}
Moreover, we construct a cut-off function $\eta(y)\in C^{\infty}(\mathbb{R})$ such that
\begin{equation}\label{weiteta}
\eta(y) =0~(\forall y<M), ~\eta(y)=1~(\forall y>B)~\text{and}~ 0\leq \eta'\leq \frac{2}{B-M},
\end{equation}
and the functions $f(y)$ and $g(y)$ obey the compatibility conditions:
\begin{enumerate}[(FG1)]
\item $f(B)=g(B), ~f'(B)=g'(B)$,\label{fg1}
\item $\eta\frac{(g')^2}{fg''}\leq \beta <1~\text{and}~ 2\eta'g'+\eta g -f''\geq0,~ \text{for any}~ y\in[M,B]$.\label{fg2}
\end{enumerate}
An example of the weight $\rho(y)$ shall be given in Appendix.

For the Lyapunov functional $G(t)$ with the solution $w$ of the problem \eqref{finBGP_Rx}, we have the following inequality.
\begin{lemma}
With the weight function $\rho$ having the properties (F1)-(F4), (G1)-(G3) and (FG1)-(FG2), the functional $G(t)$ satisfies the following estimate,
\begin{align}\label{bloupG}
\frac{dG}{dt}\geq\frac{2(1-\beta)}{\|\rho\|_{L^1(\mathbb{R}_{+})}}G^2
-\|\tilde{U}\|_{L^{\infty}([0,T])}(3+C_f)G
\end{align}
under the assumption \eqref{Blowupcondition1}.
\end{lemma}
\begin{proof}
Noting that $\rho(0)=\lim \limits_{y\rightarrow +\infty}\rho(y)
=\lim \limits_{y\rightarrow +\infty}\rho'(y)=0$, by integrating by parts, one deduces from \eqref{finBGP_Rx} that
\begin{equation}\label{LyaEqu}
\begin{split}
\frac{dG}{dt}=&\int_{0}^{\infty}\rho w^2dy
-\int_{0}^{\infty}\rho\partial_{y}^{-1}w\partial_yw dy
-\int_{0}^{\infty}\rho\partial_{y}^{-1}\tilde{U}\partial_yw dy
+2\int_{0}^{\infty}\rho\tilde{U}w dy\\
&+\int_{0}^{\infty}\rho\int^{\infty}_y w dy'dy
+\int_{0}^{\infty}\rho\partial_{y}^2wdy\\
=&2\int_{0}^{\infty}\rho w^2dy
-\frac{1}{2}\int_{0}^{\infty}\rho''(\partial_{y}^{-1}w)^2dy
+\tilde{U}\int_{0}^{\infty}y\rho' w dy
+3\tilde{U}\int_{0}^{\infty}\rho w dy\\
&+\int_{0}^{\infty}\int_{0}^{y}\rho dy' wdy
+\int_{0}^{\infty}\rho''wdy
-\rho'(0)\tilde{U}\\
:= & \sum_{k=1}^{7} J_{k}.
\end{split}
\end{equation}

Let us estimate the right hand side of \eqref{LyaEqu} term by term. Similar to \cite{Kuk-V-W},
thanks to (F\ref{f1}), (F\ref{f3}), (G\ref{g1})-(G\ref{g3}), \eqref{weiteta} and (FG\ref{fg2}), for the terms $J_1$ and $J_2$ we have
\begin{align*}
J_1+J_2=
&2\int_{0}^{\infty}\rho w^2dy
-\frac{1}{2}\int_{0}^{\infty}\eta\rho''(\partial_{y}^{-1}w)^2dy
-\frac{1}{2}\int_{0}^{\infty}(1-\eta)f''(\partial_{y}^{-1}w)^2dy\\
\geq &2\int_{0}^{\infty}\rho w^2dy
+\frac{1}{2}\int_{M}^{B} (\eta g''-f'')(\partial_{y}^{-1}w)^2dy
-\frac{1}{2}\int_{M}^{\infty}\eta g''(\partial_{y}^{-1}w)^2dy\\
\geq &2\int_{0}^{\infty}\rho w^2dy
+\frac{1}{2}\int_{M}^{B} (2\eta' g'+\eta g''-f'')(\partial_{y}^{-1}w)^2dy
\\
&
-2\int_{B}^{\infty}\eta\frac{(g')^2}{g g''} gw^2dy
-2\int_{M}^{B}\eta \frac{(g')^2}{fg''}fw^2dy\\
\geq &2(1-\beta)\int_{0}^{\infty}\rho w^2dy\\
\geq & \frac{2(1-\beta)}{\|\rho\|_{L^1(\mathbb{R}_{+})}}G^2.
\end{align*}

For the terms $J_l~(l=3,4,5,6)$, by using the properties (F\ref{f2})-(F\ref{f3}), (G\ref{g2}) and the condition \eqref{Blowupcondition1}, these terms can be bounded as below:
\begin{align*}
J_3+J_4=&\tilde{U}\int_{0}^{\infty}y\rho' w dy
+3\tilde{U}\int_{0}^{\infty}\rho w dy\\
\geq &-\|\tilde{U}\|_{L^{\infty}([0,T])}(3+C_f)G
\end{align*}
and
\begin{align*}
J_5+J_6\geq \int_{A}^{B}(\int_{0}^{A}\rho dy'+f'')wdy\geq 0.
\end{align*}

Combining the above estimates, it arrives at the inequality \eqref{bloupG} by noting
 that $J_7\geq0$ from the condition \eqref{Blowupcondition1}.
\end{proof}

\textbf{Proof of Theorem \ref{blowup}. }
From the inequality \eqref{bloupG}, we know that there exists a time $0<t^{*}\leq T$ such that
$$\lim \limits_{t\to t^{*-}} G(t)=+\infty$$
when
 $$
G(0)>\frac{\|\rho\|_{L^1(\mathbb{R}_{+})}\|\tilde{U}\|_{L^{\infty}([0,T])}(3+C_f)}{2(1-\beta)}
\frac{1}{1-e^{-\|\tilde{U}\|_{L^{\infty}([0,T])}(3+C_f)T}},
$$
which implies that $\lim \limits_{t\to t^{*-}} \|w\|_{L^{\infty}(H_{t})}=+\infty$ with the aid of the condition \eqref{Blowupcondition3} and the construction of $\rho$.

In particular, when $U \equiv 0$, the inequality
\eqref{bloupG} simplifies into
$$\frac{dG}{dt}\geq\frac{2(1-\beta)}{\|\rho\|_{L^1(\mathbb{R}_{+})}}G^2,$$
which implies that
$G(t)$ always blows up in a finite time for any given nonzero initial value satisfying the condition \eqref{Blowupcondition1}.
Thus we conclude that the solution $w$ to the problem \eqref{finBGP_Rx} must blow up in a finite time with large enough initial value, which is a contradiction with the assumption \eqref{suppose0}. Thereby we complete the proof of Theorem \ref{blowup}. \qed

\section*{Appendix: The construction of the weight $\rho(y)$}

Inspired by \cite{Kuk-V-W}, we construct $\rho(y)\in W^{2,\infty}(\mathbb{R}_+)\cap C^1(\mathbb{R}_+)$ by the profile
\begin{equation*}
\rho(y)=
\begin{cases}
ky, 0\leq y<A,\\
ay^2+by+c, A \leq y< B,\\
\frac{1}{(y+h)^{\gamma}},   B \leq y,
\end{cases}
\end{equation*}
where the parameters $A, B, a, b, c, k, h$ and $\gamma$ will be specified later. Denote by
$$f(y)=\rho(y)\quad {\rm for}~0\leq y<B, \qquad g(y)=\frac{1}{(y+h)^{\gamma}}\quad {\rm for }~y\geq 0, h>0.$$
To guarantee $\rho\in C^{1}(\mathbb{R}_+)$, i.e. the property (FG\ref{fg1}) holds, we need the relations (R):
\begin{align*}
a&=-\frac{1}{B^2-A^2}\frac{B(1+\gamma)+h}{(B+h)^{\gamma+1}},
&b=-\frac{\gamma}{(B+h)^{\gamma+1}}+\frac{2B}{B^2-A^2}\frac{B(1+\gamma)+h}{(B+h)^{\gamma+1}},\\
c&=-\frac{A^2}{B^2-A^2}\frac{B(1+\gamma)+h}{(B+h)^{\gamma+1}},
&k=-\frac{\gamma}{(B+h)^{\gamma+1}}+\frac{2}{B+A}\frac{B(1+\gamma)+h}{(B+h)^{\gamma+1}}.
\end{align*}

Now it is left to choose the proper parameters $A$, $B$, $\gamma$ and  $h$ such that  the above properties hold.

It is easy to verify the properties (F\ref{f1}),(F\ref{f4}) and (G\ref{g1})-(G\ref{g2}) provided that $k>0$, $a<0$ and $\gamma>1$. Let $$\frac{1}{2}kA^2+a\geq 0,$$
then it implies the property (F\ref{f2})-(F\ref{f3}), especially it is sufficient to  choose
$$A>0,~B\geq \frac{1}{A^2}+A~\text{and}~h>A\gamma.$$
Obviously, the property (G\ref{g3}) holds with $\beta=\frac{2\gamma +1}{2\gamma+2}$.

To obtain the property (FG2), we construct a
cutoff function $\eta_s(z)\in C^{\infty}(\mathbb{R})$ with $0 \leq \eta'_s(z) \leq 2$, and
\begin{equation}\label{eta}
\eta_s(z)=
\begin{cases}
0, &z<0;\\
\text{smooth connection}, &0 \leq z< 1;\\
1,   &z\geq 1.
\end{cases}
\end{equation}
and define
$$\eta(y)=\eta_s(\frac{y-M}{B-M})$$
 with $M\in (A, B)$ being determined later.

For the first inequality in the property (FG\ref{fg2}), it is sufficient to choose the point $M$ and the parameter $h$ such that
$$\min\Big\{\frac{g(M)}{g(B)},\frac{g(M)}{f(M)}\Big\}\leq \beta\frac{\gamma+1}{\gamma}=\frac{2\gamma+1}{2\gamma}.$$

In fact,  one has  $$\frac{g(M)}{f(M)}\leq \frac{2\gamma+1}{2\gamma}$$
provided that $h\geq\frac{B-M}{(\frac{2\gamma+1}{2\gamma})^{\frac{1}{\gamma}}}$.

On the other hand, to achieve the inequality $\frac{g(M)}{f(M)}\leq \frac{2\gamma+1}{2\gamma}$, it is equivalent to
\begin{equation}\label{constrMain}
\frac{(B+h)^{\gamma+1}}{(M+h)^{\gamma}}\leq \frac{2\gamma+1}{2\gamma}
[\frac{(2BM-M^2-A^2)(B(1+\gamma)+h)}{B^2-A^2}-\gamma M],
\end{equation}
which can be obtained by choosing $M$ and $h$ such that
\begin{equation*}
\begin{cases}
 \frac{M}{B}\geq \frac{2}{\alpha^{\frac{1}{\gamma}}}-1,\\
\frac{M}{B}\geq 1-\sqrt{(1-(\frac{A}{B})^2)\frac{1}{2\gamma+2}},\\
h\geq(4\gamma+4)\gamma B,
\end{cases}
\end{equation*}
where $\alpha=\frac{4\gamma+1}{4\gamma}$.

Moveover, by taking $h\geq \frac{4\gamma(B^2-A^2)}{B-M}$, one has
$$-\frac{4}{B-M}\frac{\gamma}{(M+h)^{\gamma+1}} -a =\big(-\frac{4\gamma}{B-M}+\frac{B(1+\gamma)+h}{B^2-A^2}\big)\frac{1}{(B+h)^{\gamma+1}}\geq 0,$$
which guarantees the second inequality in the property (FG\ref{fg2}).

For example, if we can take $A=2$, $M=4.5$, $B=5$, $\gamma=2$, $h=400$ and $C_f=1$, and the parameters $k$,~$a,~b$ and $c$ are given in the relations (R), then the functions $f(y)$ and $g(y)$ satisfy all properties listed above.  Thereby we complete the construction of $\rho$.

\vspace{.2in}
\noindent{\bf Acknowledgments:} This research was partially supported by
National Natural Science Foundation of China (NNSFC) under Grant No. 11631008.

\end{document}